\documentclass[11pt]{amsart}

\usepackage{amssymb}
\usepackage{bm}
\usepackage[centertags]{amsmath}
\usepackage{amsfonts}
\usepackage{amsthm}
\usepackage{mathrsfs}
\usepackage[usenames]{xcolor}
\usepackage[normalem]{ulem}
\usepackage{comment}
\usepackage{enumitem}
\usepackage{greek_enumerate}
\usepackage{hyperref}
\hypersetup{pdftex,colorlinks=true,allcolors=black}
\usepackage{hypcap}
\usepackage[noabbrev]{cleveref}

\makeatletter
\def\l@subsection{\@tocline{2}{0pt}{2.5pc}{2.5pc}{}}%
\makeatother

\usepackage{hyperref}
\hypersetup{pdftex,colorlinks=true,allcolors=black}
\usepackage{hypcap}

\DeclareFontFamily{OT1}{pzc}{}
\DeclareFontShape{OT1}{pzc}{m}{it}{<-> s * [1.10] pzcmi7t}{}
\DeclareMathAlphabet{\mathpzc}{OT1}{pzc}{m}{it}

\theoremstyle{definition}
\newtheorem{theorem}{Theorem}[section]
\newtheorem{definition}[theorem]{Definition}

\newtheorem{proposition}[theorem]{Proposition}

\newtheorem{remark}[theorem]{Remark}

\newtheorem{assumption}[theorem]{Assumption}
\newtheorem{notation}[theorem]{Notation}

\newtheorem*{definition*}{Definition}

\newtheorem*{theorem*}{Theorem}

\newtheorem*{corollary*}{Corollary}

\newtheorem*{proposition*}{Proposition}

\newtheorem{problem}{Problem}

\newtheorem*{remarks}{Remarks}

\newtheorem*{claim}{Claim}

\newcommand{\N}{\mathbb{N}}
\newcommand{\Z}{\mathbb{Z}}

\newcommand{\C}{\mathbb{C}}

\newcommand{\e}{\varepsilon}

\newcommand{\trn}[1]{{\left\vert\kern-0.25ex\left\vert\kern-0.25ex\left\vert #1 
    \right\vert\kern-0.25ex\right\vert\kern-0.25ex\right\vert}}

\newcommand{\trnsmall}[1]{{\vert\kern-0.25ex\vert\kern-0.25ex\vert #1 
    \vert\kern-0.25ex\vert\kern-0.25ex\vert}}

\long\def\symbolfootnote[#1]#2{\begingroup%
\def\thefootnote{\fnsymbol{footnote}}\footnote[#1]{#2}\endgroup}

\numberwithin{equation}{section}

\makeatletter
\@namedef{subjclassname@2020}{%
  \textup{2020} Mathematics Subject Classification}
\makeatother

\allowdisplaybreaks

\begin{document}

\title[The compact operators on $c_0$ as a Calkin algebra]{The compact operators on $c_0$ as a Calkin algebra}
\dedicatory{Dedicated to Gilles Godefroy}

\author[P. Motakis]{Pavlos Motakis}

\address{Department of Mathematics and Statistics, York University, 4700 Keele Street, Toronto, Ontario, M3J 1P3, Canada}

\email{pmotakis@yorku.ca}

\author[D. Puglisi]{Daniele Puglisi}

\address{Department of Mathematics and Computer Sciences, University of Catania, Catania, 95125, Italy}

\email{dpuglisi@dmi.unict.it}

\thanks{The first author was supported by NSERC Grant RGPIN-2021-03639.}




\begin{abstract}
For a Banach space $X$, let $\mathcal{L}(X)$ denote the algebra of all bounded linear operators on $X$ and let $\mathcal{K}(X)$ denote the compact operator ideal in $\mathcal{L}(X)$. The quotient algebra $\mathcal{L}(X)/\mathcal{K}(X)$ is called the Calkin algebra of $X$, and it is denoted $\mathpzc{Cal}(X)$. We prove that the unitization of $\mathcal{K}(c_0)$ is isomorphic as a Banach algebra to the Calkin algebra of some Banach space $\mathcal{Z}_{\mathcal{K}(c_0)}$. This Banach space is an Argyros-Haydon sum $(\oplus_{n=1}^\infty X_n)_\mathrm{AH}$ of a sequence of copies $X_n$ of a single Argyros-Haydon space $\mathfrak{X}_\mathrm{AH}$, and the external versus the internal Argyros-Haydon construction parameters are chosen from disjoint sets.
\end{abstract}

\maketitle


\tableofcontents

\section{Introduction}

The Calkin algebra of a Banach space $X$ (see, e.g., \cite{calkin:1941}, \cite{yood:1954}, or \cite{caradus:pfaffenberger:yood:1974}) is defined as $\mathpzc{Cal}(X) = \mathcal{L}(X)/\mathcal{K}(X)$, where $\mathcal{L}(X)$ denotes the algebra of bounded linear opetors on $X$ and $\mathcal{K}(X)$ denotes its ideal of compact operators. Given a unital Banach algebra $\mathcal{B}$, showing that it is a Calkin algebra means finding a Banach space $X$ such that $\mathpzc{Cal}(X)$ is isomorphic to $\mathcal{B}$ as a Banach algebra. This is a challenging problem; solving it demands designing $X$ having in mind the algebraic structure of the indirectly defined object $\mathcal{L}(X)$ and the position of $\mathcal{K}(X)$ in it. The starting point of this field of study is the Argyros-Haydon space $\mathfrak{X}_\mathrm{AH}$ (see \cite{argyros:haydon:2011}) with $\mathpzc{Cal}(\mathfrak{X}_\mathrm{AH})$ one-dimensional, i.e., it has the scalar-plus-compact property. The construction of $\mathfrak{X}_\mathrm{AH}$ relies on immensely influential works by Tsirelson (\cite{tsirelson:1974}), Bourgain and Delbaen (\cite{bourgain:delbaen:1980}), Schlumprecht (\cite{schlumprecht:1991}), and Gowers and Maurey (\cite{gowers:maurey:1993}). Following \cite{argyros:haydon:2011}, several examples of explicit unital Banach algebras $\mathcal{B}$ have been shown to be Calkin algebras. Some of them are the convolution algebra $\ell_1(\N_0)$ (\cite{tarbard:2013}) and all separable commutative $C^*$-algebras (\cite{motakis:puglisi:zisimopoulou:2016} and \cite{motakis:2024}). Other examples include the Hilbert space $\ell_2$ with coordinate-wise multiplication induced by an orthonormal basis (this is the first known reflexive infinite dimensional Calkin algebra) (\cite{motakis:pelczar:2024}, see also \cite{pelczar-barwacz:2023}) and the quasireflexive James space with coordinate-wise multiplication induced by its standard shrinking basis (\cite{motakis:puglisi:tolias:2020}). The aforementioned examples are commutative; non-commutative ones also exist, but are confined to spaces of the type $\mathcal{B} = \oplus_{i=1}^n M_{k_i}(\mathcal{B}_i)$ for some $\mathcal{B}_i$ from the previous lists (\cite[Page 1022, Note added in proof]{kania:laustsen:2017} and \cite[Corollary 8.8]{motakis:2024}). The endeavour of finding more involved explicit non-commutative Calkin algebras was sparked by a set of questions of N. C. Phillips focusing primarily on $C^*$-algebras (see \cite[Section 9]{motakis:2024}). We contribute to this by proving the following.

\begin{theorem*}
There exists a Banach space $\mathcal{Z}_{\mathcal{K}(c_0)}$ with Calkin algebra isomorphic to the unitization of $\mathcal{K}(c_0)$ as a Banach algebra.
\end{theorem*}

The Banach space $\mathcal{Z}_{\mathcal{K}(c_0)}$ is a Bourgain-Delbaen-$\mathscr{L}_\infty$ sum  of Bourgain-Delbaen-$\mathscr{L}_\infty$-spaces. Bourgain-Delbaen-$\mathscr{L}_\infty$-spaces (introduced in \cite{bourgain:delbaen:1980}, see also \cite{argyros:gasparis:motakis:2016}) comprise an important class that has played a major role in Banach space theory. Prominent examples are the original Bourgain-Delbaen space from \cite{bourgain:delbaen:1980}, and the Argyros-Haydon space from \cite{argyros:haydon:2011}, but there are many others. For example, in \cite{kurka:2019}, such spaces were used to negatively solve Godefroy's problem from \cite{godefroy:2010} on whether the isomorphism class of $c_0$ is Borel. Another application is from \cite{causey:2020} where it was proved that such a space from \cite{argyros:gasparis:motakis:2016} is a counterexample to a conjecture of Godefroy, Kalton, and Lancien from \cite{godefroy:kalton:lancien:2001} stating an $\ell_1$-predual with summable Slenk should be isomorphic to $c_0$. In this paper, we use Bourgain-Delbaen-$\mathscr{L}_\infty$ sums from \cite{zisimopoulou:2014} similarly to \cite{motakis:puglisi:zisimopoulou:2016}. The Banach space $\mathcal{Z}_{\mathcal{K}(c_0)}$ is a special type of Bourgain-Delbaen-$\mathscr{L}_\infty$-sum: it is an Argyros-Haydon-sum $(\oplus_{n=1}^\infty X_n)_\mathrm{AH}$, and each $X_n$ is a copy of a single Argyros-Haydon space $\mathfrak{X}_\mathrm{AH}$. The classical Argyros-Haydon construction uses as parameter a sequence of pairs or natural numbers $(m_j,n_j)_{j=1}^\infty$ satisfying certain growth conditions. Here, the space $\mathfrak{X}_\mathrm{AH}$ has parameter a sequence of pairs $(m_j,n_j)_{j\in L_1}$ whereas the external Argyros-Haydon has parameter a sequence of pairs $(m_j,n_j)_{j\in L_2}$, and the sets $L_1$ and $L_2$ are disjoint infinite subsets of $\N$. As such, our construction relies on directly applying known techniques and does not involve radical modifications of heavy machinery.

We postpone the details of the construction of $\mathcal{Z}_{\mathcal{K}(c_0)}$ and proof of its main properties until Sections \ref{AH sums section}, \ref{AH sums operators section}, and \ref{ZKC0 section}. We initially focus on the principles that yield that its Calkin algebra is the unitization of $\mathcal{K}(c_0)$. Recall that a {\em Schauder decomposition} of a Banach space $X$ is a sequence of closed subspaces $(X_n)_{n=1}^\infty$ of $X$ such that each $x\in X$ admits a unique representation $x = \sum_{n=1}^\infty x_n$, where $x_n\in X_n$, for all $n\in\N$. This induces a sequence of bounded linear projections $P_n:X\to X$, $n\in\N$, where $P_nx = \sum_{k=1}^nx_k$ and $C = \sup_n\|P_n\|$ is finite. If, in particular, $C=1$, then the Schauder decomposition is called {\em monotone}. If $\cup_{n=1}^\infty P_n^*(X^*)$ is dense in $X^*$ then we call $(X_n)_{n=1}^\infty$ {\em shrinking}. For every interval $I$ of $\N$ we also let $P_Ix = \sum_{k\in I}x_k$.

Given a Banach space $X$ with a Schauder decomposition $(X_n)_{n=1}^\infty$, every bounded linear operator $T:X\to X$ admits a matrix representation $(T_{m,n})_{m,n=1}^\infty$, where $T_{m,n}:X_n\to X_m$, $m,n\in\N$. This can be used to approximate $T$ in the strong operator topology. We define the algebra $\mathcal{HA}(X)$ of {\em horizontally approximable} operators on $X$ comprising the $T:X\to X$ for which this approximation is valid in the operator norm topology. When $(X_n)_{n=1}^\infty$ is shrinking, $\mathcal{K}(X)\subset\mathcal{HA}(X)$, and thus, $\mathcal{HA}(X)/\mathcal{K}(X)$ is a subalgebra of $\mathpzc{Cal}(X)$. We observe that if $X$ has the scalar-plus-compact property, then $\mathcal{HA}(c_0(X))/\mathcal{K}(c_0(X))$ coincides with $\mathcal{K}(c_0)$. We exploit this fact as follows. We introduce an equivalence condition between Schauder decompositions $(X_n)_{n=1}^\infty$ and $(Y_n)_{n=1}^\infty$ of Banach spaces $X$ and $Y$ respectively, called {\em finite Fredholm equivalence}. If it is satisfied, then $\mathcal{HA}(X)/\mathcal{K}(X)$ and $\mathcal{HA}(Y)/\mathcal{K}(Y)$ are isomorphic Banach algebras, even if $X$ and $Y$ are not isomorphic Banach spaces. We prove that in particular the standard Schauder decomposition of a Bourgain-Delbaen-$\mathscr{L}_\infty$-sum of a sequence of Banach spaces $(X_n)_{n=1}^\infty$ is finitely Fredholm equivalent to the standard Schauder decomposition of $(\oplus_{n=1}^\infty X_n)_0$. This is directly applicable to the following distinguishing features of $\mathcal{Z}_{\mathcal{K}(c_0)}$, equipped with its standard Schauder decomposition:
\begin{enumerate}[label=(\roman*),leftmargin=19pt]
    
    \item It is a Bourgain-Delbaen-$\mathscr{L}_\infty$ sum of countably many copies of a space $\mathfrak{X}_\mathrm{AH}$ with the scalar-plus compact property.

    \item Every bounded linear operator on $\mathcal{Z}_{\mathcal{K}(c_0)}$ is a scalar multiple of the identity plus a horizontally approximable operator.
    
\end{enumerate}
By the preceding discussion, $\mathcal{HA}(\mathcal{Z}_{\mathcal{K}(c_0)})/\mathcal{K}(\mathcal{Z}_{\mathcal{K}(c_0)})$ is isomorphic as a Banach algebra to $\mathcal{HA}(c_0(\mathfrak{X}_\mathrm{AH}))/\mathcal{K}(c_0(\mathfrak{X}_\mathrm{AH}))$, which, by the first item, coincides with $\mathcal{K}(c_0)$. The second property yields that $\mathcal{HA}(\mathcal{Z}_{\mathcal{K}(c_0)})/\mathcal{K}(\mathcal{Z}_{\mathcal{K}(c_0)})$ is of codimension-one in $\mathpzc{Cal}(X)$.

\section{Horizontally approximable operators}
In this section, we study the algebra $\mathcal{HA}(X)$ of horizontally approximable operators on a Banach space $X$ with a Schauder decomposition $(X_n)_{n=1}^\infty$. This is a subcollection of the algebra of horizontally compact operators, introduced in \cite[Definition 7.1]{zisimopoulou:2014} (see also \Cref{HC definition}). An operator $T:X\to X$ is horizontally approximable if the finite initial submatrices of its standard representation as an infinite matrix $(T_{m,n})_{m,n=1}^\infty$ with respect to $(X_n)_{n=1}^\infty$ can be used to approximate $T$ in operator norm. When $(X_n)_{n=1}^\infty$ is shrinking, $\mathcal{K}(X)\subset\mathcal{HA}(X)$, and thus, $\mathcal{HA}(X)/\mathcal{K}(X)$ is a subalgebra of the Calkin algebra of $X$. For example, if $X$ is an infinite dimensional Banach space with the scalar-plus-compact property, we prove that $\mathcal{HA}(c_0(X))/\mathcal{K}(c_0(X))$ coincides with $\mathcal{K}(c_0)$. We then provide an equivalence condition between Schauder decompositions $(X_n)_{n=1}^\infty$ and $(Y_n)_{n=1}^\infty$ of Banach spaces $X$ and $Y$ respectively. When it is satisfied, $\mathcal{HA}(X)/\mathcal{K}(X)$ and $\mathcal{HA}(Y)/\mathcal{K}(Y)$ are isomorphic as Banach algebras, even if $X$ and $Y$ are not isomorphic.

\begin{definition}
Let $X$ be a Banach space with a Schauder decomposition $(X_n)_{n=1}^\infty$ and associated sequence of projections $(P_n)_{n=1}^\infty$.
\begin{enumerate}[leftmargin=19pt,label=(\alph*)]
    
\item A bounded linear operator $A$ on $X$ is called {\em horizontally approximable} if any of the following equivalent conditions is satisfied.\vskip6pt
\begin{enumerate}[label=(\greek*)]
    \item $A = \lim_n AP_n$ and  $A = \lim_nP_nA$,\vskip6pt
    \item $A = \lim_m\lim_nP_mAP_n$,\vskip6pt
    \item $A = \lim_n\lim_mP_mAP_n$,\vskip6pt
    \item $A = \lim_nP_nAP_n$.\vskip6pt
\end{enumerate}
where all limits are with respect to the operator norm topology.

\item The collection of horizontally approximable bounded linear operators on $X$ is denoted $\mathcal{HA}(X)$.

\end{enumerate}
The definitions of horizontally approximable operators and $\mathcal{HA}(X)$ depend on the Schauder decomposition $(X_n)_{n=1}^\infty$, but this will always be fixed, and thus, we suppress it in the notation.
\end{definition}

\begin{remark}
Let $X$ be a Banach space with a Schauder decomposition $(X_n)_{n=1}^\infty$. We will identify a $T\in\mathcal{L}(X)$ with its standard matrix representation $(T_{m,n})_{m,n=1}^\infty$ with respect to $(X_n)_{n=1}^\infty$, i.e., for $m,n\in\N$, $T_{m,n} = P_{\{m\}}TP_{\{n\}}:X\to X$. Letting ``$\mathrm{sot}$'' denote the strong operator topology of $\mathcal{L}(X)$, the following hold.
\begin{enumerate}[leftmargin=19pt,label=(\roman*)]
    
    \item If $T = (T_{m,n})_{m,n=1}^\infty$ then
\[T =\mathrm{sot}\text{-}\sum_{m=1}^\infty\mathrm{sot}\text{-}\sum_{n=1}^\infty T_{m,n} = \mathrm{sot}\text{-}\sum_{n=1}^\infty\mathrm{sot}\text{-}\sum_{m=1}^\infty T_{m,n}.\]

\item If $T = (T_{m,n})_{m,n=1}^\infty$ and $S = (S_{m,n})_{n=1}^\infty$ then
\[TS = \Big(\mathrm{sot}\text{-}\sum_{k=1}^\infty T_{m,k}S_{k,n}\Big)_{m,n=1}^\infty.\]

\end{enumerate}
We also sometimes naturally identify each $T_{m,n}$ with a map $T_{m,n}:X_n\to X_m$ by restricting its domain and codomain. It will always be clear from context which map the notation $T_{m,n}$ refers to.
\end{remark}

\begin{remarks}
Let $X$ be a Banach space with a Schauder decomposition $(X_n)_{n=1}^\infty$.
\begin{enumerate}[leftmargin=21pt,label=(\roman*)]

    \item  $\mathcal{HA}(X)$ is a closed subalgebra of $\mathcal{L}(X)$ and
    \[\mathcal{HA}_{00}(X) = \cup_{n=1}^\infty\big\{A\in\mathcal{L}(X): A = P_nAP_n\big\}\]
    is dense in $\mathcal{HA}(X)$. An operator $T = (T_{m,n})_{m,n=1}^\infty\in\mathcal{L}(X)$ is in $\mathcal{HA}_{00}(X)$ if and only if all but finitely many of its entries are zero.

    \item An operator $T = (T_{m,n})_{m,n=1}^\infty\in\mathcal{L}(X)$ is in $\mathcal{HA}(X)$ if and only if
    \[T = \sum_{m=1}^\infty\sum_{n=1}^\infty T_{m,n} = \sum_{n=1}^\infty\sum_{m=1}^\infty T_{m,n},\]
where the convergence is in the operator norm topology.
    
    \item If $A = (A_{m,n})_{m,n=1}^\infty$, $B = (B_{m,n})_{m,n=1}^\infty$ are in $\mathcal{HA}(X)$ then
    \[AB = \Big(\sum_{k=1}^\infty A_{m,k}B_{k,n}\Big)_{m,n=1}^\infty,\]
where the convergence is in the operator norm topology.

    \item If we additionally assume that the Schauder decomposition $(X_n)_{n=1}^\infty$ is shrinking, $\mathcal{K}(X)\subset \mathcal{HA}(X)$. In particular, $\mathcal{HA}(X)/\mathcal{K}(X)$ is a closed subalgebra of $\mathpzc{Cal}(X)$.

\end{enumerate}
\end{remarks}

\begin{notation}
For a sequence of Banach spaces $(X_n)_{n=1}^\infty$ we denote
    \[(\oplus_{n=1}^\infty X_n)_{0} = \Big\{(x_n)_{n=1}^\infty\in \prod_{n=1}^\infty X_n:\lim_n\|x_n\| = 0\Big\}\]
with $\|(x_n)_{n=1}^\infty\| = \max_n\|x_n\|$. If $X$ is a Banach space and, for all $n\in\N$, $X_n = X$, then we denote this space $c_0(X)$.
\end{notation}

\begin{remark}
For a sequence of Banach spaces $(X_n)_{n=1}^\infty$, the space $Y = (\oplus_{n=1}^\infty X_n)_{0}$ admits a standard shrinking Schauder decomposition $(\hat X_n)_{n=1}^\infty$ where, for $n\in\N$,
\[\hat X_n = \Big\{(x_m)_{m=1}^\infty\in Y: x_m = 0\text{ for }m\neq n\Big\}.\]
\end{remark}

\begin{remark}
We identify the Banach space $c_0(\ell_1)$ with all infinite scalar matrices $A = (a_{m,n})_{m,n=1}^\infty$ such that
\begin{gather}
\|A\| = \sup_m\sum_{n=1}^\infty|a_{m,n}|<\infty\text{ and}\\
\lim_m\Big(\sum_{n=1}^\infty|a_{m,n}|\Big) = 0.
\end{gather}
This is then naturally isometrically identified with $\mathcal{K}(c_0)$ by letting for $A = (a_{m,n})_{m,n=1}^\infty$, $Ae_n = \sum_{m=1}^\infty a_{m,n}e_m$. With this identification, usual matrix multiplication corresponds to composition of operators, and the collection of matrices $A$ with finitely many non-zero entries is dense in $\mathcal{K}(c_0)$.
\end{remark}

\begin{proposition}
\label{horizontally approximable by compact in c0 of scalar-plus-compact}
Let $X$ be an infinite dimensional Banach space with the scalar-plus-compact property and consider $c_0(X)$ with its standard shrinking Schauder decomposition. Then, $\mathcal{HA}(c_0(X))/\mathcal{K}(c_0(X))$ is isometrically isomorphic to $\mathcal{K}(c_0)$ as a Banach algebra.
\end{proposition}

\begin{proof}
Denote $Y=c_0(X)$. For $A = (a_{m,n})_{m,n=1}^\infty$ in $\mathcal{K}(c_0)$ with all but finitely many zero entries we define $\hat A = (a_{m,n}I_X)_{m,n=1}^\infty\in\mathcal{HA}_{00}(Y)$, i.e., for $(x_n)_{n=1}^\infty\in Y$,
\[\hat A (x_n)_{n=1}^\infty = \Big(\sum_{n=1}^\infty a_{m,n}x_n\Big)_{m=1}^\infty.\]
Let us verify that $\|\hat A\|\leq \|A\|$. Indeed, if $(x_n)_{n=1}^\infty$ is in the unit ball of $Y$, then
\[\big\|\hat A\big((x_n)_{n=1}^\infty\big)\big\| = \max_m\Big\|\sum_{n=1}^\infty a_{m,n}x_n\Big\| \leq \max_m\sum_{n=1}^\infty |a_{m,n}| = \|A\|.\]
It is then entirely straightforward to check that $A\mapsto \hat A$ extends to a non-expansive linear homomorphism from $\mathcal{K}(c_0)$ to $\mathcal{L}(Y)$. If we denote $\mathcal{Q}:\mathcal{HA}(Y)\to \mathcal{HA}(Y)/\mathcal{K}(Y)$ the quotient map, then $A\mapsto \mathcal{Q}\hat A$ is such a map as well. We will use the infinite-dimensionality of $X$ to also easily obtain $\|A\| \leq \|\mathcal{Q}\hat A\|$. Indeed, let $A = (a_{m,n})_{m,n=1}^\infty\in\mathcal{K}(c_0)$ with all but finitely many zero entries and choose $M,N\in\N$ such that, for all $(x_n)_{n=1}^\infty\in Y$,
\[\hat A\big((x_n)_{n=1}^\infty\big) = \Big(\sum_{n=1}^Na_{1,n}x_n,\ldots,\sum_{n=1}^Na_{M,n}x_n,0,0,\ldots\Big).\]
Let $K:Y\to Y$ be a compact operator and fix $\e>0$. For every subspace $Z$ of $X$ and $n\in\N$ let
\[\hat Z(n) = \big\{(0,\ldots0,x_n,0,\ldots):x_n\in Z\big\},\]
which is a subspace of $Y$. By the compactness of $K$, we may inductively choose infinite-dimensional subspaces $Z_1\supset Z_2\supset\cdots\supset Z_N = Z$ of $X$ such that, for $n=1,\ldots,N$, $\|K|_{\hat Z(n)}\|<\e/N$. Then, for every $x\in Z$ of norm one and signs $(\e_n)_{n=1}^N$,
\[\Big\|\big(\hat A-K\big)\big((\e_1x,\e_2x,\ldots,\e_Nx,0,\ldots)\big)\Big\| \geq \max_{1\leq m\leq M}\Big|\sum_{n=1}^N\e_na_{m,n}\Big| - \e.\]
In particular, $\|\hat A - K\| \geq \|A\| - \e$. Thus $A\mapsto \mathcal{Q}\hat A$ is an isometric linear homomorphism.

To show that $A\mapsto \mathcal{Q}\hat A$ is onto, let $B = (S_{m,n})_{m,n=1}^\infty\in\mathcal{HA}_{00}(Y)$. Because $X$ has the scalar-plus-compact property, for $m,n\in\N$, there are a scalar $a_{m,n}$ and a compact operator $K_{m,n}:X\to X$ such that $S_{m,n}=a_{m,n}I_X + K_{m,n}$ and $a_{m,n}$, $K_{m,n}$ are non-zero for for only finitely many $m,n\in\N$. Consider $A = (a_{m,n})_{m,n=1}^\infty\in\mathcal{K}(c_0)$ and note that $\hat A - B\in\mathcal{K}(Y)$, i.e., $\mathcal{Q}\hat A = \mathcal{Q}B$. By the density of $\mathcal{HA}_{00}(Y)$ in $\mathcal{HA}(Y)$, $A\mapsto \mathcal{Q}\hat A$ is onto.
\end{proof}

\begin{notation}
For Banach spaces $X$ and $Y$ and bounded linear operators $T,S:X\to Y$ we will write $T\sim S$ to mean $T-S\in\mathcal{K}(X,Y)$.
\end{notation}

\begin{definition}
\label{finitely-Fredhold-equivalent}
Let $X$ and $Y$ be Banach spaces with Schauder decompositions $(X_n)_{n=1}^\infty$, $(Y_n)_{n=1}^\infty$ respectively and associated projections $(P_n)_{n=1}^\infty$, $(Q_n)_{n=1}^\infty$ respectively.
We call $(X_n)_{n=1}^\infty$ and $(Y_n)_{n=1}^\infty$ {\em $C$-finitely Fredholm equivalent}, for some $C\geq 1$, if there exist bounded linear operators $L_n:X\to Y$ and $R_n:Y\to X$, $n\in\N$, that satisfy
\begin{enumerate}[label=(\greek*)]
    
    \item\label{finitely-Fredhold-equivalent-a} the {\em uniform boundedness condition}
    \[\big(\sup_n\|L_n\|\big)\big(\sup_n\|R_n\|\big) \leq C,\]

    \item\label{finitely-Fredhold-equivalent-b} the {\em compatibility condition}
    \begin{align*}
    L_m\sim L_nP_m\sim Q_mL_n\text{ and }R_m\sim R_nQ_m\sim P_mR_n,
    \end{align*}
    for $m\leq n\in\N$,
    \item\label{finitely-Fredhold-equivalent-c} and the {\em initial Fredholm invertibility condition}
    \begin{align*}
    R_nL_n\sim P_n\text{ and }L_nR_n\sim Q_n,
    \end{align*}
    for $n\in\N$.
\end{enumerate}
\end{definition}

\begin{proposition}
\label{hor-app-by-compact isomorphism}
Let $X$ and $Y$ be Banach spaces with $C$-finitely Fredholm equivalent shrinking Schauder decompositions, for some $C\geq 1$. Then, $\mathcal{HA}(X)/\mathcal{K}(X)$ and $\mathcal{HA}(Y)/\mathcal{K}(Y)$ are $C^2$-isomorphic as Banach algebras.
\end{proposition}

\begin{proof}
Adopt the notation of \Cref{finitely-Fredhold-equivalent} and let $\mathcal{Q}_Y:\mathcal{L}(Y)\to\mathcal{L}(Y)/\mathcal{K}(Y)$ and $\mathcal{Q}_X:\mathcal{L}(X)\to\mathcal{L}(X)/\mathcal{K}(X)$ denote the quotient maps. For $n\in\N$ and $A\in\mathcal{HA}(X)$ we define
\[\phi_n(A) = L_nAR_n.\]
We note that, by \cref{finitely-Fredhold-equivalent-b}, $\phi_n(A) \sim \phi_n(P_nAP_n) \sim Q_n\phi_n(A)Q_n$, and thus,
\[\phi_n(A)\in \{B\in\mathcal{L}(Y):B = Q_nBQ_n\}+\mathcal{K}(Y)\subset\mathcal{HA}(Y).\]
Note that, for $A\in\mathcal{HA}_{00}(X) = \cup_{n=1}^\infty\{A\in\mathcal{L}(X):A=P_nAP_n\}$, the sequence $(\mathcal{Q}_Y\phi_n(A))_{n=1}^\infty$ is eventually constant. Indeed, if $m\in\N$ is such that $A = P_mAP_m$ then, for $n\geq m$, by \cref{finitely-Fredhold-equivalent-b},
\[\phi_n(A) = L_nAR_n = L_nP_mAP_mR_n\sim L_m AR_m = \phi_m(A).\]
Additionally, for $A,B\in\mathcal{HA}_{00}(X)$ and $n\in\N$ sufficiently large, $\mathcal{Q}_Y\phi_n(AB) = \mathcal{Q}_Y\phi_n(A)\mathcal{Q}_Y\phi_n(B)$. Indeed, if $n\in\N$ is such that $A=P_nAP_n$ and $B = P_nBP_n$ then, by \cref{finitely-Fredhold-equivalent-c},
\[\phi_n(AB) = L_nABR_n = L_nAP_nBR_n\sim L_nAR_nL_nBP_n = \phi_n(A)\phi_n(B).\]
We conclude that, by \cref{finitely-Fredhold-equivalent-a}, that the formula $A \mapsto \lim_n\mathcal{Q}_Y(L_nAR_n)$ defines a bounded homomorphism $\Phi:\mathcal{HA}(X)\to \mathcal{HA}(Y)/\mathcal{K}(Y)$ of norm at most $C$, and $\mathcal{K}(X)\subset \mathrm{ker}(\Phi)$. Let $\widehat\Phi:\mathcal{HA}(X)/\mathcal{K}(X)\to\mathcal{HA}(Y)/\mathcal{K}(Y)$ be the unique bounded homomorphism such that, for $A\in\mathcal{HA}(X)$, $\widehat \Phi(\mathcal{Q}_XA) = \Phi(A)$.

By interchanging the roles of $X$ and $Y$, the formula $B\mapsto \lim_n\mathcal{Q}_X(R_nBL_n)$ defines a bounded homomorphism $\Psi:\mathcal{HA}(Y)\to \mathcal{HA}(X)/\mathcal{K}(X)$ of norm at most $C$, and $\mathcal{K}(Y)\subset\mathrm{ker}(\Psi)$. Let $\widehat\Psi:\mathcal{HA}(Y)/\mathcal{K}(Y)\to\mathcal{HA}(X)/\mathcal{K}(X)$ be the unique bounded homomorphism such that, for $B\in\mathcal{HA}(Y)$, $\widehat \Psi(\mathcal{Q}_YB) = \Psi(B)$. We claim that $\hat\Phi$ and $\hat\Psi$ are each other's inverse, and to that end, by density, it suffices to prove
\[
\begin{split}
\hat \Psi\big(\hat\Phi(\mathcal{Q}_XA)\big) &= \mathcal{Q}_XA,\text{ for }A\in\mathcal{HA}_{00}(X)\text{ and}\\
\hat \Phi\big(\hat\Psi(\mathcal{Q}_YB)\big) &= \mathcal{Q}_YB,\text{ for }B\in\mathcal{HA}_{00}(Y).
\end{split}
\]
We only prove the former because the arguments to show each of these are identical. Let $A\in\mathcal{HA}_{00}(X)$, such that $\Phi(A) = \lim_n\mathcal{Q}_Y\phi_n(A)$. Recall that, if $m\in\N$ is such that $A = P_mAP_m$, then, for $n\geq m$, $\mathcal{Q}_Y\phi_n(A) = \mathcal{Q}_Y\phi_m(A)$. Therefore, $\Phi(A) =\mathcal{Q}_Y\phi_m(A)$ and
\[
\begin{split}
\hat \Psi\big(\hat\Phi(\mathcal{Q}_XA)\big) &= \hat \Psi\big(\Phi(A)\big) = \Psi\big(\phi_m(A)\big)=  \Psi(L_mAR_m)\\
&=\lim_n\mathcal{Q}_X(R_nL_mAR_mL_n).
\end{split}
\]
For $n\geq m$, by \ref{finitely-Fredhold-equivalent-b} and \ref{finitely-Fredhold-equivalent-c}, $R_nL_m\sim R_nQ_mL_m \sim R_mL_m \sim P_m$. Similarly, $R_mL_n\sim P_m$, and thus, $\hat \Psi\big(\hat\Phi(\mathcal{Q}_XA)\big) = \mathcal{Q}_X(P_mAP_m) = \mathcal{Q}_X(A)$.
\end{proof}

We remark that, under additional regularity assumptions on $Y$, $C$-finite Fredholm equivalence can be improved to 1-finite Fredholm equivalence, after renorming $X$.

\begin{proposition}
\label{general renorming}
Let $X$ and $Y$ be Banach spaces with $C$-finitely Fredholm equivalent Schauder decompositions $(X_n)_{n=1}^\infty$, $(Y_n)_{n=1}^\infty$ respectively, for some $C\geq 1$. Adopting the notation of \Cref{finitely-Fredhold-equivalent}, assume that $(Y_n)_{n=1}^\infty$ is monotone and, for every $m,n\in\N$, $L_m R_n = Q_{\min\{m,n\}}$. Then, after passing to a $C$-equivalent norm of $X$ but maintaining the original norm of $Y$, $(X_n)_{n=1}^\infty$, $(Y_n)_{n=1}^\infty$ are 1-finitely Fredholm equivalent.
\end{proposition}

\begin{proof}
Let $\|\cdot\|_X$, $\|\cdot\|_Y$ denote the original norms of $X$ and $Y$ respectively and put $\alpha = \sup_n\|R_n:(X,\|\cdot\|_X)\to (Y,\|\cdot\|_Y)\|$. Define the equivalent norm on $X$ given by
\[\trn{x}_X = \max\big\{\alpha^{-1}\|x\|_X,\sup_n\|L_nx\|_Y\big\}.\]
It is easy to check that $\trn{\cdot}_X$ is $C$-equivalent to $\|\cdot\|_X$. Then, by definition, for every $n\in\N$, $\|L_n:(X,\trn{\cdot}_X)\to (Y,\|\cdot\|_Y)\|\leq 1$. Also, for $m\in\N$ and $y\in Y$ such that $\|y\|_Y\leq 1$ we will show $\trn{R_my}_X\leq 1$. Indeed,
\[\alpha^{-1}\|R_my\|_X \leq \alpha^{-1}\|R_m:(Y,\|\cdot\|_Y)\to(X,\|\cdot\|_X)\|\|y\|_Y\leq 1\]
and, for $n\in\N$,
\[\|L_n(R_my)\|_Y = \|Q_{\min\{n,m\}}y\|_Y \leq 1.\]
This means that the Schauder decompositions of $(X,\trn{\cdot}_X)$ and $(Y,\|\cdot\|_Y)$ are 1-finitely Fredholm equivalent.
\end{proof}

\section{Bourgain-Delbaen-$\mathscr{L}_\infty$ sums of Banach spaces}
\label{abstract Bourgain-Delbaen sums}
We recall the notion of an abstract $C$-Bourgain-Delbaen-$\mathscr{L}_\infty$ sum $\mathcal{Z}$ of of a sequence of Banach spaces $(X_n)_{n=1}^\infty$, introduced in \cite[Definition 2.3]{zisimopoulou:2014}. This space $\mathcal{Z}$ admits a Schauder decomposition $(\mathcal{Z}_n)_{n=1}^\infty$ that we prove is $C$-finitely Fredholm equivalent to the standard Schauder decomposition of $(\oplus_{n=1}^\infty X_n)_0$. The easy consequence of this is that if each $X_n$ is a copy of a common Banach space $X$ with the scalar-plus-compact property and $\mathcal{Z}$ satisfies the scalar-plus-horizontally approximable property, then the Calkin algebra of such a $\mathcal{Z}$ is isomorphic, as a Banach algebra, to the unitization of $\mathcal{K}(c_0)$. The goal of the subsequent Sections \ref{AH sums section}, \ref{AH sums operators section}, and \ref{ZKC0 section} is to gather all the necessary details from \cite{argyros:haydon:2011} and \cite{zisimopoulou:2014}, and to prove that such $\mathcal{Z}$ exists.

\begin{definition}
\label{abstract BD sum}
Let $(X_n)_{n=1}^\infty$ be a sequence of Banach spaces and $(\Delta_n)_{n=1}^\infty$ be a sequence of disjoint finite sets. Let
\[\mathcal{W} = \big(\oplus_{n=1}^\infty\big(X_n\oplus\ell_\infty(\Delta_n)\big)_\infty\big)_\infty\]
denote the Banach space of all sequences of pairs $(x_n,y_n)_{n=1}^\infty$ such that, for $n\in\N$, $(x_n,y_n)\in X_n\times \ell_\infty(\Delta_n)$, with
\begin{equation*}
\|(x_n,y_n)_{n=1}^\infty\|_\mathcal{W} = \sup_{n\in\N}\max\{\|x_n\|,\|y_n\|\} <\infty.
\end{equation*}
For $n\in\N$, we canonically identify
\[(X_n\oplus\ell_\infty(\Delta_n))_\infty\text{ and }\mathcal{W}_n = \big(\oplus_{k=1}^n\big(X_k\oplus\ell_\infty(\Delta_k)\big)_\infty\big)_\infty\]
with subspaces of $\mathcal{W}$ and let
\[r_n:\mathcal{W}\to \mathcal{W}_n\]
denote the canonical restriction map. For $n\in\N$, let $\pi_{X_n}:\mathcal{W}\to X_n$ denote the canonical quotient map.

Assume that $C\geq1$ and $i_n:\mathcal{W}_n\to\mathcal{W}$, $n\in\N$ are linear operators that satisfy the following.
\begin{enumerate}[label=(\greek*)]
    \item\label{BD assumption a} For all $n\in\N$, $\|i_n\|\leq C$.
    \item\label{BD assumption b} For all $n\in\N$, $r_ni_n:\mathcal{W}_n\to\mathcal{W}_n$ is the identity map.
    \item\label{BD assumption c} For all $m\leq n\in\N$, $i_nr_ni_m = i_m:\mathcal{W}_m\to\mathcal{W}$.
    \item\label{BD assumption d} For all $m<n\in\N$, $\pi_{X_n}i_m:\mathcal{W}_m\to\mathcal{W}$ is the zero map.
\end{enumerate}
For $n\in\N$, define $\mathcal{Z}_n = i_n(X_n\oplus\ell_\infty(\Delta_n))_\infty$ and $\mathcal{Z} = \overline{\langle\cup_{n=1}^\infty\mathcal{Z}_n\rangle}$. The space $\mathcal{Z}$ is called a $C$-Bourgain-Delbaen-$\mathscr{L}_\infty$ sum of $(X_n)_{n=1}^\infty$.
\end{definition}

\begin{proposition}
\label{standard BD decomposition}
For a $C$-Bourgain-Delbaen-$\mathscr{L}_\infty$ sum $\mathcal{Z}$ of a sequence of Banach spaces $(X_n)_{n=1}^\infty$ the following hold.
\begin{enumerate}[label=(\roman*),leftmargin=19pt]

    \item\label{standard BD decomposition i} $(\mathcal{Z}_n)_{n=1}^\infty$ is a Schauder decomposition of $\mathcal{Z}$ with associated projections $P_n = i_nr_n|_{\mathcal{Z}}:\mathcal{Z}\to\mathcal{Z}$, $n\in\N$.

    \item\label{standard BD decomposition ii} For all $n\in\N$, $P_n(\mathcal{Z}) = i_n(\mathcal{W}_n)$.
    
\end{enumerate}

\end{proposition}

\begin{proof}
It is clear, by \cref{BD assumption a} of Definition \ref{abstract BD sum}, that $\sup\|P_n\|\leq C$ and, by items \ref{BD assumption b} and \ref{BD assumption c} of Definition \ref{abstract BD sum}, that $P_mP_n = P_{\min\{m,n\}}$. We will show:
\begin{enumerate}[label=($\mathfrak{\alph*}$)]
    \item For $n\leq m\in\N$, $P_m(\mathcal{Z}_n) = \mathcal{Z}_n$.
    \item For $m<n\in\N$, $P_m(\mathcal{Z}_n) = \{0\}$.
\end{enumerate}
\[z = i_n\big((0,0),\ldots,(0,0),(x_n,y_n)\big).\]
These easily yield \cref{standard BD decomposition i}. For the former, by the definition of $\mathcal{Z}_n$, $P_n(\mathcal{Z}_n) = \mathcal{Z}_n$, and thus, for $m\geq n$,
\[P_m(\mathcal{Z}_n) = P_m\big(P_n(\mathcal{Z}_n)\big) = (P_mP_n)(\mathcal{Z}_n) = P_n(\mathcal{Z}_n) = \mathcal{Z}_n.\]
For the latter, let $z\in\mathcal{Z}_n$, i.e, for some $x_n\in X_n$ and $y_n\in\ell_\infty(\Delta_n)$. Then, for $m<n$,
\[
\begin{split}
P_{m}z &= i_{m}r_{m} i_n\big((0,0),\ldots,(0,0),(x_n,y_n)\big)\big)\\
&= i_{m} r_{m}\big((0,0),\ldots,(0,0),(x_n,y_n)\big)\big)\\
&= i_{m}\big((0,0),\ldots,(0,0)\big) = 0.
\end{split}
\]

For the proof of \cref{standard BD decomposition ii}, the inclusion $P_n(\mathcal{Z})\subset i_n(\mathcal{W}_n)$ is obvious. We prove, by induction on $n\in\N$, that $i_n(\mathcal{W}_n)\subset P_n(\mathcal{Z}_n)$. For $n=1$, this is obvious. Let now $n\geq 2$ such that $i_{n-1}(\mathcal{W}_n) \subset P_{n-1}(\mathcal{Z}_n)$ and let
\[z = i_n\big((x_1,y_1),\ldots,(x_{n-1},y_{n-1}),(x_n,y_n)\big)\in i_n(\mathcal{W}_n)\]
Put
\[w = i_{n-1}\big((x_1,y_1),\ldots,(x_{n-1},y_{n-1})\big),\]
which, by the inductive hypothesis, is in $P_{n-1}(\mathcal{Z}) = \mathcal{Z}_1+\cdots+\mathcal{Z}_{n-1}$. By items \ref{BD assumption c} and \ref{BD assumption d} of \Cref{abstract BD sum} we may write
\[w = i_n\big((x_1,y_1),\ldots,(x_{n-1},y_{n-1}),(0,w_n)\big),\]
thus,
\[z - w = i_n\big((0,0),\ldots,(0,0),(x_n,y_n-w_n)\big)\in \mathcal{Z}_n.\]
In conclusion, $z\in \mathcal{Z}_1+\cdots+\mathcal{Z}_{n-1}+\mathcal{Z}_n = P_n(\mathcal{Z})$.
\end{proof}

\begin{proposition}
\label{BD sum finitely fredholm equivalent to c0 sum}
Let $\mathcal{Z}$ be a $C$-Bourgain-Delbaen-$\mathscr{L}_\infty$ sum of a sequence of Banach spaces $(X_n)_{n=1}^\infty$. The following hold.
\begin{enumerate}[label=(\roman*),leftmargin=19pt]
    
    \item\label{BD sum finitely fredholm equivalent to c0 sum-i} $(\mathcal{Z}_n)_{n=1}^\infty$ is $C$-finitely Fredholm equivalent to the standard Schauder decomposition of $(\oplus_{n=1}^\infty X_n)_{c_0}$.

    \item\label{BD sum finitely fredholm equivalent to c0 sum-ii} $\mathcal{Z}$ admits a $C$-equivalent norm such that $(\mathcal{Z}_n)_{n=1}^\infty$ is $1$-finitely Fredholm equivalent to the standard Schauder decomposition of $(\oplus_{n=1}^\infty X_n)_{c_0}$.
    
\end{enumerate}

\end{proposition}

\begin{proof}
Denote $Y = (\oplus_{n=1}^\infty X_n)_{c_0}$ and, for $n\in\N$, define the following $L_n:\mathcal{Z}\to Y$ and $R_n:Y\to\mathcal{Z}$.
\begin{gather*}
    L_n\big((x_k,y_k)_{k=1}^\infty\big) = (x_1,x_2,\ldots,x_n,0,0,\ldots)\\
    R_n\big((x_k)_{k=1}^\infty\big) = i_n\big((x_1,0),(x_2,0),\ldots,(x_n,0)\big).
\end{gather*}
and note that $\sup_{n \in \mathbb{N}}\|L_n\|\|R_n\|\leq C$, i.e., the uniform boundedness condition is satisfied. We will also show that the compatibility and initial Fredholm invertibility conditions are satisfied. But first, let us note that $L_mR_n = Q_{\min\{m,n\}}$ holds trivially. Thus, \cref{BD sum finitely fredholm equivalent to c0 sum-ii} will follow immediately from \cref{BD sum finitely fredholm equivalent to c0 sum-i}.

To show the compatibility condition, fix $m\leq n\in\N$. It is obvious that $L_m = Q_mL_n$. To show $L_nP_m = L_m$, let $(x_k,y_k)_{k=1}^\infty\in \mathcal{Z}$ and note 
\[
\begin{split}
P_m\big((x_k,y_k)_{k=1}^\infty\big) &= i_m\big((x_1,y_1),\ldots,(x_m,y_m)\big)\\
&= \big((x_1,y_1),\ldots,(x_m,y_m),(0,w_m+1),(0,w_{m+2}),\ldots\big),
\end{split}
\]
by items \ref{BD assumption b} and \ref{BD assumption d}, of \Cref{abstract BD sum}. This easily yields
\[L_nP_m\big((x_k,y_k)_{k=1}^\infty\big) = L_m\big((x_k,y_k)_{k=1}^\infty\big).\]
To show $R_n = P_mR_n$, let $(x_k)_{k=1}^\infty\in Y$. Then
\[
\begin{split}
P_mR_n\big((x_k)_{k=1}^\infty\big) &= i_mr_m\big(i_n\big((x_1,0),\ldots,(x_m,0),\ldots,(x_n,0)\big)\big)\\
&=i_m\big(\big((x_1,0),\ldots,(x_m,0)\big)\big)\text{ (by \ref{BD assumption b} of \Cref{abstract BD sum})}\\
&=R_m\big((x_k)_{k=1}^\infty\big).
\end{split}
\]
To prove $R_nQ_m\sim R_m$ we define the auxiliary bounded linear operator $A_{m,n}:\mathcal{Z}\to\mathcal{Z}$ such that, for $(x_k,y_k)_{k=1}^\infty\in\mathcal{Z}$.
\[A_{m,n}\big((x_k,y_k)_{k=1}^\infty\big) = i_n\big((0,0),\ldots,(0,0),(0,y_{m+1}),\ldots,(0,y_n)\big),\]
which is finite rank because of factoring through the finite-dimensional space $\ell_\infty(\Delta_{m+1})\oplus\cdots\oplus\ell_\infty(\Delta_n)$. We will show $R_m - A_{m,n}R_m = R_nQ_m$. Indeed, for $(x_k)_{k=1}^\infty\in Y$, note that, by items \ref{BD assumption d} and \ref{BD assumption c} of Definition \ref{abstract BD sum},
\[
\begin{split}
R_m\big((x_k)_{k=1}^\infty\big) &= i_m\big((x_1,0),\ldots,(x_m,0)\big)\\
&= \big((x_1,0),\ldots,(x_m,0),(0,y_{m+1}),(0,y_{m+2}),\ldots\big)\\
&= i_n\big((x_1,0),\ldots,(x_m,0),(0,y_{m+1}),\ldots,(0,y_n)\big).
\end{split}
\]
Therefore,
\[
\begin{split}
\big(R_m - A_{m,n}R_m\big)\big((x_k)_{k=1}^\infty\big) &= i_n\big((x_1,0),\ldots,(x_m,0),(0,0),\ldots,(0,0)\big)\\
&= R_nQ_m\big((x_k)_{k=1}^\infty\big),
\end{split}
\]
and thus, $R_nQ_m\sim R_m$.

We prove the initial Fredholm invertibility condition and fix $n\in\N$. Note that $R_nL_n = Q_n$ is rather obvious and that, for $(x_k,y_k)_{k=1}^\infty\in\mathcal{Z}$,
\[L_nR_n\big((x_k,y_k)_{k=1}^\infty\big) = i_n\big((x_1,0),\ldots,(x_n,0)\big),\]
and thus,
\[\big(P_n - L_nR_n\big)\big((x_k,y_k)_{k=1}^\infty\big) = i_n\big((0,y_1),\ldots,(0,y_n)\big).\]
Therefore $P_n - L_nR_n$ is a finite rank operator because it factors through the finite-dimensional space $\ell_1(\Delta_1)\oplus\cdots\oplus\ell_\infty(\Delta_n)$.
\end{proof}

We formalize the scalar-plus-horizontally approximable property of a Banach space with a Schauder decomposition.

\begin{definition}
A Banach space $X$ with a Schauder decomposition $(X_n)_{n=1}^\infty$ is said to have the {\em scalar-plus-horizontally approximable property} if every bounded linear operator $T$ on $X$ is of the form $\lambda I_X + A$, for some scalar $\lambda$ and $A\in\mathcal{HA}(X)$, where $I_X:X\to X$ is the identity operator.
\end{definition}

Note that if $X$ has the scalar-plus-horizontally approximable property, then the Schauder decomposition $(X_n)_{n=1}^\infty$ is shrinking.

The following is the culmination of Sections \ref{AH sums section}, \ref{AH sums operators section}, and \ref{ZKC0 section}. Specifically, its statement is a combination of \Cref{XAH has scalar-plus-compact} and \Cref{ZKC0 has scalar-plus-horizontally approximable}.

\begin{theorem}
\label{AH sum of AH}
There exists a Banach space $\mathfrak{X}_\mathrm{AH}$ with the scalar-plus-compact property and a 2-Bourgain-Delbaen-$\mathscr{L}_\infty$-sum $\mathcal{Z}_{\mathcal{K}(c_0)}$ of countably many copies of $\mathfrak{X}_\mathrm{AH}$ that has the scalar-plus-horizontally approximable property.
\end{theorem}

The following is the main result of this paper. Given the existence of the space from \Cref{AH sum of AH}, its proof is a direct consequence of the tools and language developped so far.

\begin{theorem}
The Calkin algebra of $\mathcal{Z}_{\mathcal{K}(c_0)}$ is isomorphic to the unitization of $\mathcal{K}(c_0)$ as a Banach algebra. More precisely, after passing to a $2$-equivalent norm of $\mathcal{Z}_{\mathcal{K}(c_0)}$, $\mathpzc{Cal}(\mathcal{Z}_{\mathcal{K}(c_0)})$ contains a two-sided ideal of codimension one that is isometrically isomorphic to $\mathcal{K}(c_0)$ as a Banach algebra.
\end{theorem}

\begin{proof}
Because $\mathcal{Z}_{\mathcal{K}(c_0)}$ has the scalar-plus-horizontally approximable property, $\mathcal{HA}(\mathcal{Z}_{\mathcal{K}(c_0)})/\mathcal{K}(\mathcal{Z}_{\mathcal{K}(c_0)})$ is a two-sided ideal of codimension one in the space $\mathpzc{Cal}(\mathcal{Z}_{\mathcal{K}(c_0)})$. By \Cref{BD sum finitely fredholm equivalent to c0 sum} and  \Cref{hor-app-by-compact isomorphism}, after passing to a 2-equivalent norm of $\mathcal{Z}_{\mathcal{K}(c_0)}$,  $\mathcal{HA}(\mathcal{Z}_{\mathcal{K}(c_0)})/\mathcal{K}(\mathcal{Z}_{\mathcal{K}(c_0)})$ is isometrically isomorphic to $\mathcal{HA}(c_0(X))/\mathcal{K}(c_0(X))$ which, by \Cref{horizontally approximable by compact in c0 of scalar-plus-compact}, is isometrically isomorphic to $\mathcal{K}(c_0)$.
\end{proof}

\begin{remark}
The above theorem does not state that, after passing to an equivalent norm, the Calkin algebra of $\mathcal{Z}_{\mathcal{K}(c_0)}$ coincides isometrically with the natural unitization of $\mathcal{K}(c_0)$ inside $\mathcal{L}(c_0)$, i.e.,  $\C I_{c_0}\oplus\mathcal{K}(c_0)$.
\end{remark}

\section{Argyros-Haydon sums of a sequence of Banach spaces}
\label{AH sums section}
In this section we repeat the definition of an Argyros-Haydon sum of a sequence of Banach space $(X_n)_{n=1}^\infty$ from \cite[Sections 4 and 5]{zisimopoulou:2014}. Although the presented definition is essentially the same, there are some mild, mostly notational, differences that we briefly elaborate on in \Cref{skip norming functionals} and in the paragraph before \Cref{scalar-plus-horizontally compact zisimopoulou}. Specifically, instead of considering a sequence of Banach spaces $X_n$, $n\in\N$, each with a norming sequence $(f_i^n)_{i=1}^\infty$, we consider a sequence of closed subspaces $(X_n)_{n=1}^\infty$ of $\ell_\infty$, each with the standard norming sequence of the coordinate functionals on $\ell_\infty$.

Let $(m_j)_{j=1}^\infty$ and $(n_j)_{j=1}^\infty$ be increasing sequences of positive integers.
\begin{assumption}
\label{lacunary assumption}
We assume that $(m_j)_{j=1}^\infty$, $(n_j)_{j=1}^\infty$ satisfy the following:
\begin{enumerate}[label=(\greek*)]
    
    \item $m_1\geq 4$.

    \item $m_{j+1}\geq m_j^2$.

    \item $n_1\geq m_1^2$.

    \item $n_{j+1}\geq (16n_j)^{\log_2(m_{j+1})}$.

\end{enumerate}
\end{assumption}
The sequence of pairs $(m_j,n_j)_{j=1}^\infty$ will be referred to as a set of parameters because it is used to define a sequence of disjoint finite sets $(\Delta_n)_{n=1}^\infty$ that depend on it. This definition also includes additional information encoded into a finite collection of partial functions from $\Gamma = \cup_{n=1}^\infty\Delta_n$ to some sets. All this information is then used to define, for a sequence $(X_n)_{n=1}^\infty$ of closed subspaces of $\ell_\infty$, a specific Bourgain-Delbaen-$\mathscr{L}_\infty$-sum
\[\mathcal{Z} = \mathcal{Z}\big((m_j,n_j)_{j=1}^\infty,(X_n)_{n=1}^\infty\big)\subset \mathcal{W} = \big(\oplus_{n=1}^\infty(X_n\oplus\ell_\infty(\Delta_n))_\infty\big)_\infty\]
of $(X_n)_{n=1}^\infty$. We will refer to such a space $\mathcal{Z}$ as an Argyros-Haydon sum of $(X_n)_{n=1}^\infty$ with parameters $(m_j,n_j)_{j=1}^\infty$.

\subsection{Definition of the sets $(\Delta_n)_{n=1}^\infty$}
We construct disjoint finite sets $\Delta_n$, $n\in\N$, depending on the pair of sequences $(m_j)_{j=1}^\infty$ and $(n_j)_{j=1}^\infty$ alongside a finite collection of partial functions from $\Gamma = \cup_{n=1}^\infty\Delta_n$ to some sets.

Let $\Delta_1^0 = \emptyset$, $\Delta_1^1 = \{(1)\}$, and $\Delta_1 = \Delta_1^0\cup\Delta_1^1$. For $\gamma = (1)$ we let $\mathrm{rank}(\gamma) = 1$ and $\sigma(\gamma) = 2$. Assume we have defined disjoint sets $\Delta_1,\ldots,\Delta_n$ such that, for $1\leq k\leq n$, $\Delta_k$ is the disjoint union of two sets $\Delta_k^0$ and $\Delta_k^1$. Denote $\Gamma_n = \cup_{k=1}^n\Delta_n$ and $\Gamma^1_n = \cup_{k=1}^n\Delta^1_n$. Assume that, for $\gamma\in\Gamma_1^1$ we have defined $\mathrm{rank}(\gamma)$, $\mathrm{cut}(\gamma)$, and $\mathrm{age}(\gamma)$ in $\{1,\ldots,n\}$, $\mathrm{weight}(\gamma)\in\{m_1^{-1},\ldots,m_{2(n-1)}^{-1}\}$, and $\mathrm{base}(\gamma)$, $\mathrm{top}(\gamma)$ in sets that will be specified in the inductive step. Assume we have also defined an injection $\sigma:\Gamma^1_n\to\N$ such that, for $\gamma\in\Gamma_n^1$, $\sigma(\gamma)>\mathrm{rank}(\gamma)$.

We fix $N(n)\in\N$ sufficiently large, to be determined soon. Define
\begin{align*}
G_n &= \big\{pe^{i2\pi q}:p,q\in[0,1]\cap\mathbb{Q},\;p(N(n)!)\in\Z,\text{ and }q(N(n)!)\in\Z\big\}\\
K_n& = \Big\{f = \big((f_k(i))_{i=1}^n\big)_{k=1}^n\in\prod_{k=1}^nG_n^n:\sum_{k=1}^n\sum_{i=1}^n|f_k(i)|\leq 1\Big\},\text{ and}\\
B_{n} &= \Big\{b^* = \big((b_k^*(\gamma))_{\gamma\in\Delta_k}\big)_{k=1}^n\in\prod_{k=1}^nG_n^{\Delta_k}:\sum_{k=1}^n\sum_{\gamma\in\Delta_k}|b_k^*(\gamma)| \leq 1\Big\}.
\end{align*}
For $1\leq k\leq n$ and $\xi\in\Delta_k$, we let $e_\xi^*$ denote the member $\big((b_k^*(\gamma))_{\gamma\in\Delta_k}\big)_{k=1}^n$ of $B_n$ such that $b^*_k(\xi) = 1$ and it has he value zero in all other entries.

Note that $G_n$ is a $4/(N(n)!)$-net of the complex unit disk, and therefore, $B_{n}$ is a $(7|\Gamma_n|)/(N(n)!)$-net of the unit ball of $(\oplus_{r=1}^n\ell_1(\Delta_r))_1$. We demand that $N(n)$ has been chosen sufficiently large such that $B_{n}$ is a $1/2^n$-net in the unit ball of $(\oplus_{k=1}^n\ell_1(\Delta_k))_1$. Similarly, we demand that $K_n$ is a $1/2^n$ net in the unit ball of $(\oplus_{k=1}^n\ell_1^n)_1$. We also demand $N(n)\geq N(n-1)$, and thus $G_n\supset G_{n-1}$, (which is the sole purpose of the factorial).

We define
\[\Delta_{n+1}^0 = \{n+1\}\times K_n.\]
For $\gamma\in\Delta_{n+1}^0$, we use the following notation.
\begin{enumerate}[label=(\arabic*)]

\item Let $\mathrm{rank}(\gamma) = n+1$ denote the first coordinate of $\gamma$.

\item Let $\mathrm{top}(\gamma) = f\in K_n$ denote the second coordinate of $\gamma$.

\end{enumerate}

To define $\Delta_{n+1}^1$, we first consider the set
\[\Delta_{n+1}^{1,\mathrm{aux}} = \{n+1\}\times \big(\{\emptyset\}\cup\Gamma^1_n\setminus\Gamma_1^1\big)\times\{m^{-1}_1,\ldots,m^{-1}_{2n}\}\times \{1,\dots,n-1\}\times B_n.\]
For $\gamma\in\Delta_{n+1}^{1,\mathrm{aux}}$, we use the following notation.
\begin{enumerate}[label=(\arabic*),resume]
    
    \item Let $\mathrm{rank}(\gamma) = n+1$ denote the first coordinate of $\gamma$.
    
    \item Let $\mathrm{base}(\gamma)$ denote the second coordinate of $\gamma$ and note
    \begin{itemize}
        \item either $\mathrm{base}(\gamma)=\xi\in\Gamma^1_n\setminus\Gamma_1^1$
        \item or $\mathrm{base}(\gamma) = \emptyset$, in which case we say it is undefined.
    \end{itemize}   
    
    \item Let $\mathrm{weight}(\gamma)\in\{m^{-1}_1,\ldots,m^{-1}_{2n}\}$ denote the third coordinate of $\gamma$.

    \item Let $\mathrm{cut}(\gamma) = p\in\{1\,\ldots,n-1\}$ denote the fourth coordinate of $\gamma$.

    \item Let $\mathrm{top}(\gamma) = b^*\in B_n$ denote the last coordinate of $\gamma$.

    \item Let $\mathrm{age}(\gamma)$ be defined as follows.
    \begin{itemize}
        \item If $\mathrm{base}(\gamma)$ is undefined let $\mathrm{age}(\gamma) = 1$.
        \item If $\mathrm{base}(\gamma) = \xi$ let $\mathrm{age}(\gamma) =\mathrm{age}(\xi)+ 1$.
    \end{itemize}        
\end{enumerate}

We define $\Delta_{n+1}^{1,\mathrm{even}}$ as all $\gamma\in\Delta_{n+1}^{1,\mathrm{aux}}$ such that, for some $1\leq j\leq n$, $\mathrm{weight}(\gamma) = m_{2j}^{-1}$ and
\begin{enumerate}[label=(\greek*),leftmargin=19pt]
    \item\label{even no base} either $\mathrm{base}(\gamma)$ is undefined or,
    \item\label{even with base} if $\mathrm{base}(\gamma) = \xi\in\Gamma_n^1\setminus\Delta_1^1$, then $\mathrm{weight}(\xi) = m_{2j}^{-1}$ and $\mathrm{age}(\gamma)< n_{2j}$.
\end{enumerate}
We define $\Delta_{n+1}^{1,\mathrm{odd}}$ as all $\gamma\in\Delta_{n+1}^\mathrm{aux}$ such that, for some $1\leq j\leq n$, $\mathrm{weight}(\gamma) = m_{2j-1}^{-1}$ and
\begin{enumerate}[label=(\greek*),leftmargin=19pt,resume]
    \item\label{odd no base} either $\mathrm{base}(\gamma)$ is undefined and $\mathrm{top}(\gamma) = e_\eta^*$, for some $\eta\in\Gamma^1$ satisfying $\mathrm{weight}(\eta) = m^{-1}_{4i-2} < m_{2j-1}^{-2}$ or,
    \item\label{odd with base} if $\mathrm{base}(\gamma) = \xi\in\Gamma_n^1\setminus\Gamma_1^1$, then $\mathrm{weight}(\xi) = m_{2j-1}^{-1}$ and $\mathrm{age}(\gamma)< n_{2j}$, and $\mathrm{top}(\gamma) = e_\eta^*$, for some $\eta\in\Gamma^1$ satisfying $\mathrm{weight}(\eta) = m_{4\sigma(\xi)}$.
\end{enumerate}
Define $\Delta_{n+1}^1 = \Delta_{n+1}^{1,\mathrm{even}}\cup \Delta_{n+1}^{1,\mathrm{odd}}$, $\Delta_{n+1} = \Delta_{n+1}^0\cup\Delta_{n+1}^1$ and, letting $\Gamma_{n+1}^1 = \cup_{k=1}^{n+1}\Gamma_k^1$, extend $\sigma$ to an injective function $\sigma:\Gamma^1_{n+1}\to \N$ such that, for $\gamma\in\Gamma_{n+1}^1$, $\sigma(\gamma)>\mathrm{rank}(\gamma)$.

We let $\Gamma = \cup_{k=1}^\infty\Delta_k$ and note the dependence on $(m_j)_{j=1}^\infty$ and $(n_j)_{j=1}^\infty$. We also denote $\Gamma^0 = \cup_{k=1}^\infty\Delta_k^0$ and $\Gamma^1 = \cup_{k=1}^\infty\Delta_k^1$.

\begin{remark}
\label{skip norming functionals}
In \cite{zisimopoulou:2014}, the set $\Gamma$ does not only depend on $(m_j,n_j)_{j=1}^\infty$, but also on a sequence of Banach spaces $(X_n)_{n=1}^\infty$ and it incorporates in its Definition some information about them, via some norming functionals on them. We have replaced this with the sets $K_n$, $n\in\N$, which encode the usual action of absolute convex combinations of the usual norming functionals $(e_i^*)_{i=1}^\infty$ on $\ell_\infty$, and thus on all its subspaces.
\end{remark}

\subsection{Definition of the space $\mathcal{Z} = \mathcal{Z}\big((m_j,n_j)_{j=1}^\infty,(X_n)_{n=1}^\infty\big)$}

For $n\in\N$, let $X_n$ be a closed subspace of $\ell_\infty$. Denote
\[\mathcal{W} = \big(\oplus_{n=1}^\infty\big(X_n\oplus\ell_\infty(\Delta_n)\big)_\infty\big)_\infty.\]
We will inductively define linear extension maps $i_{n,n+1}:\mathcal{W}_n\to\mathcal{W}_{n+1}$ (i.e., $r_n i_{n,n+1}:\mathcal{W}_{n}\to \mathcal{W}_{n}$ is the identity map), such that additionally $\pi_{X_{n+1}} i_{n,n+1} = 0$. Assuming that we have defined $i_{1,2}$,\ldots,$i_{n-1,n}$, we will define $i_{n,n+1}:\mathcal{W}_n\to\mathcal{W}_{n+1}$ by first specifying  linear functionals $c_\gamma^*:\mathcal{W}\to\C$, $\gamma\in\Delta_{n+1}$, and, for $w = \big((x_1,y_1),\ldots,(x_n,y_n)\big)\in\mathcal{W}_n\subset\mathcal{W}$ put
\[i_{n,n+1}(w) = \Big((x_1,y_1),\ldots,(x_n,y_n),\big((0,(c_\gamma^*(w))_{\gamma\in\Delta_{n+1}}\big)\Big).\]
Denote, for $m<n$, $i_{m,n} = i_{n-1,n}\circ i_{n-2,n-1}\circ\cdots\circ i_{m,m+1}:\mathcal{W}_m\to\mathcal{W}_n$ and $i_{n,n}:\mathcal{W}_n\to\mathcal{W}_n$ the identity map. We canonically identify each $b^* = ((b_k^*(\gamma))_{\gamma\in\Delta_k})_{k=1}^n$ in $B_n$ with with a linear functional $b^*:\mathcal{W}\to\C$ given by
\begin{equation}
\label{delta-one tops}
b^*\big((x_k,y_k)_{k=1}^\infty\big) = \sum_{k=1}^n\langle b_k^*,y_k\rangle
\end{equation}
and each $f = \big((f_k(i))_{i=1}^n\big)_{k=1}^n\in K_n$, is identified with a linear functional $f:\mathcal{W}\to\C$ given by
\begin{equation}
\label{delta-zero tops}
f\big((x_k,y_k)_{k=1}^\infty\big) = \sum_{k=1}^n\sum_{i=1}^nf_k(i)e_i^*(x_k),
\end{equation}
where $(e_i^*)_{i=1}^\infty$ are the standard coordinate functionals of $\ell_\infty$.

 Let $\gamma\in\Delta_{n+1}$.
 \begin{enumerate}[label=(\roman*),leftmargin=21pt]
    
     \item If $\gamma\in\Delta_{n+1}^0$ with $\mathrm{top}(\gamma) = f\in K_n$, using the identification \eqref{delta-zero tops}, define
     \[c_\gamma^* = f.\]
     
     \item If $\gamma\in\Delta_{n+1}^1$ is as in \ref{even no base} or \ref{odd no base}, with $\mathrm{weight}(\gamma) = m_j^{-1}$, $\mathrm{cut}(\gamma) = p\in\{1,\ldots,n-1\}$, and $\mathrm{top}(\gamma) = b^*\in B_n$, using \eqref{delta-one tops}, define
     \[c_\gamma^* = \frac{1}{m_j}\big(b^*-b^*i_{p,n}r_p\big).\]

     \item If $\gamma\in\Delta_{n+1}^1$ is as in \ref{even with base} or \ref{odd with base}, with $\mathrm{base}(\gamma) = \xi\in\Gamma_n^1$, $\mathrm{weight}(\gamma) = m_j^{-1}$, $\mathrm{cut}(\gamma) = p\in\{1,\ldots,n-1\}$, and $\mathrm{top}(\gamma) = b^*\in B_n$, using \eqref{delta-one tops}, define
    \[c_\gamma^* = e_\xi^*+\frac{1}{m_j}\big(b^*-b^*i_{p,n}r_p\big).\]
    
 \end{enumerate}

By \cite[Proposition 5.1]{zisimopoulou:2014}, for every $m\leq n$, $\|i_{m,n}\|\leq 2$. We may therefore, for $n\in\N$, define $i_n:\mathcal{W}_n\to\mathcal{W}$ such that, for $z = \big((x_1,y_1),\ldots,(x_n,y_n)\big)\in\mathcal{W}_n$,
\[i_n(z) = \big((x_1,y_1),\ldots,(x_n,y_n),(0,z_{n+1}),(0,z_{n+2}),\ldots\big),\]
where, for $k\geq n$,
\[r_{k+1}i_n(z) = i_{k,k+1}r_{k}(z) = i_{k,k+1}\big((x_1,y_1),\ldots,(x_n,y_n),(0,z_{n+1}),\ldots,(0,z_{k})\big).\]
Given this, it is not hard to see that the assumptions of \Cref{abstract BD sum} are satisfied, and thus, letting, for $n\in\N$, $\mathcal{Z}_n = i_n(X_n\oplus\ell_\infty(\Delta_n))_\infty$, the space
\[\mathcal{Z}\big((m_j,n_j)_{j=1}^\infty,(X_n)_{n=1}^\infty\big) =  \overline{\langle\cup_{n=1}^\infty\mathcal{Z}_n\rangle}\]
is a 2-Bourgain-Delbaen-$\mathscr{L}_\infty$ sum of $(X_n)_{n=1}^\infty$. We refer to this space as the Argyros-Haydon sum of $(X_n)_{n=1}^\infty$ with parameters $(m_j,n_j)_{j=1}^\infty$.

The role of the set $\Gamma^0$ is the following.

\begin{remark}
For every $x\in\mathcal{Z}$ and $\e>0$ there exists $\gamma\in\Gamma$ such that $|e_\gamma^*(x)| \geq \|x\|-\e$. This was, essentially, shown in \cite[Lemma 5.7]{zisimopoulou:2014}, but we repeat the brief argument. Let $z = (x_k,y_k)_{k=1}^\infty\in\mathcal{Z}$ and note $\|z\| = \max\{\sup_k\|x_k\|,\sup_k\|y_k\|\}$. If this maximum is obtained by the second term, then the conclusion is obvious. Otherwise assume that, for some $k_0\in\N$, $\|z\| <\|x_{k_0}\| + \e/2$. Pick $i_0\in\N$ such that $|e_{i_0}^*(x_{k_0})| > \|x_{k_0}\| - \e/2$. Let $n=\max\{k_0, i_0\}$ and $f = ((f_k(i))_{i=1}^n)_{k=1}^n$ such that $f_{k_0}(i_0) = 1$, and all other entries are zero. Then, for $\gamma = (n+1,f)\in\Delta_{n+1}^0$, $|e_\gamma^*(z)| = |e_{i_0}^*(x_{k_0})| > \|z\|-\e$.
\end{remark}

The following is proved in \cite[Corollary 5.15]{zisimopoulou:2014}.

\begin{theorem}
\label{shrinking and duality}
The Schauder decomposition $(\mathcal{Z}_n)_{n=1}^\infty$ of the space $\mathcal{Z} = \mathcal{Z}\big((m_j,n_j)_{j=1}^\infty,(X_n)_{n=1}^\infty\big)$ is shrinking. In particular, $\mathcal{Z}^*$ is $2$-isomorphic to $\big(\oplus_{n=1}^\infty(X_n^*\oplus\ell_1(\Delta_n))_1\big)_1$.
\end{theorem}

\subsection{Evaluation analysis of $e_\gamma^*$}
The evaluation analysis of a coordinate functional $e_\gamma^*$ is a central concept invented in \cite{argyros:haydon:2011}, and it has been used in all Argyros-Haydon constructions.

For $\gamma\in\Delta_1$ let $c_\gamma^* = 0$, and for $\gamma\in\Gamma$, let $d_\gamma^* = e_\gamma^* - c^*_\gamma$. For $\gamma\in\Delta_1$, $e_\gamma^* = d_\gamma^*$ while, for $\gamma\in\Gamma^0$ with $\mathrm{top}(\gamma) = f\in K_n$, it directly follows $e_\gamma^* = d_{\gamma}^*+f$.

The following can be proved easily by induction on the age of a $\gamma$ (see, e.g., \cite[Proposition 5.3]{zisimopoulou:2014} or \cite[Proposition 4.5]{argyros:haydon:2011}).

\begin{proposition}[Evaluation analysis]
For $\gamma\in\Gamma^1$ with $\mathrm{rank}(\gamma)>1$,
\[e_\gamma^* = \sum_{r=1}^ad_{\xi_r}^* + \frac{1}{m_j}\sum_{r=1}^ab_r^*\circ P_{(p_r,q_r]},\]
where $a=\mathrm{age}(\gamma)$, $m_j^{-1}=\mathrm{weigh}(\gamma)$, $\xi_a = \gamma$, for $1\leq r<a$, $\xi_r = \mathrm{base}(\xi_{r+1})$, and, for $1\leq r\leq a$, $b_r^*=\mathrm{top}(\xi_r)$, $p_r = \mathrm{cut}(\xi_r)$, $q_r = \mathrm{rank}(\xi_r)-1$. Furthermore, for $1\leq r\leq a$, $p_r<q_r$ and, for $1\leq r<a$, $q_r+1<p_{r+1}$. 
\end{proposition}

It is crucial to have an abundance of available coordinate functionals to achieve lower bounds for linear combinations of certain vectors. This is the purpose of the following, and it was proved in \cite[Proposition 4.7]{argyros:haydon:2011}.

\begin{proposition}
\label{coordinate builder}
Let $j\in\N$, $1\leq a\leq n_{2j}$, and $(p_r,q_r)_{r=1}^a$ be pairs of positive integers such that $2j\leq p_r<q_r$, $1\leq r\leq a$, and $p_{r-1}+1<q_r$, $2\leq r\leq a$. Let also $b_r^*\in B_{q_r}$, $1\leq r\leq a$. Then, there exist $\xi_r\in\Delta^1_{q_r+1}$, $1\leq r\leq a$,  and $\gamma\in\Gamma^1_{q_a+1}$ that has evaluation analysis
\[e_\gamma^* = \sum_{r=1}^ad_{\xi_r}^* + \frac{1}{m_{2j}}\sum_{r=1}^ab_r^*\circ P_{(p_r,q_r]}.\]
\end{proposition}

\section{Operators on Argyros-Haydon sums of Banach spaces}

Here we recall one of the main results from \cite{zisimopoulou:2014}, namely that under appropriate assumptions a space $\mathcal{Z} = \mathcal{Z}\big((m_j,n_j)_{j=1}^\infty,(X_n)_{n=1}^\infty\big)$ satisfies the scalar-plus-horizontally compact property. We then study conditions under which an operator $T:\mathcal{Z}^{(1)}\to\mathcal{Z}^{(2)}$, where $\mathcal{Z}^{(1)}$, $\mathcal{Z}^{(2)}$ are two different Argyros-Haydon sums, is ``horizontally small''.

\label{AH sums operators section}

\begin{definition}
\label{HC definition}
Let $X$ be a Banach space with a Schauder decomposition $(X_n)_{n=1}^\infty$ and associated projections $(P_n)_{n=1}^\infty$.
\begin{enumerate}[label=(\alph*),leftmargin=19pt]
    
    \item A sequence $(x_k)_{k=1}^\infty$ is called horizontally block if there exist successive intervals $I_k$, $k\in\N$, of $\N$ such that, for all $k\in\N$, $x_k = P_{I_k}x_k$.

    \item For a Banach space $Y$, a bounded linear operator $T:X\to Y$ is called {\em horizontally compact} if any of the following equivalent conditions holds:
    \begin{enumerate}[label=(\greek*)]

        \item Whenever $(x_k)_{k=1}^\infty$ is a bounded horizontally block sequence in $X$, $\lim_k\|Tx_k\| = 0$.

        \item $T = \lim_nTP_n$, in operator norm.

    \end{enumerate}

\end{enumerate}

\end{definition}

The following, although not stated in this more general form, is proved in \cite[Proposition 7.8]{zisimopoulou:2014}. There, the assumption that every $T:\mathcal{Z}\to X_n$ is horizontally compact is realized by demanding that each $X_n$ either has the Schur property, or $X_n^*$ contains no isomorphic copy of $\ell_1$ (see \cite[Proposition 3.6]{zisimopoulou:2014}). We note that in \cite[page 21]{zisimopoulou:2014}, the author made additional assumptions about the norming functionals $(f_i^n)_{i=1}^\infty$ associated to each space $X_n$, $n\in\N$ (see the introduction of Section \ref{AH sums section} and \Cref{skip norming functionals}). These additional assumptions are not a requirement in the proof and they were only used to simplify a perturbation argument as in \cite{argyros:haydon:2011}.

\begin{theorem}[\cite{zisimopoulou:2014}]
\label{scalar-plus-horizontally compact zisimopoulou}
    Let $(m_j,n_j)_{j=1}^\infty$ satisfy \Cref{lacunary assumption}, $(X_n)_{n=1}^\infty$ be a sequence of closed subspaces of $\ell_\infty$, and denote $\mathcal{Z} = \mathcal{Z}\big((m_j,n_j)_{j=1}^\infty,(X_n)_{n=1}^\infty\big)$. Assume that, for every $n\in\N$, every bounded linear operator $T:\mathcal{Z}\to X_n$ is horizontally compact. Then, for every bounded linear operator $T:\mathcal{Z}\to\mathcal{Z}$, there exists $\lambda\in\C$ such that $T-\lambda I_\mathcal{Z}$ is horizontally compact.
\end{theorem}

\begin{notation}
Let $(m_j,n_j)_{j=1}^\infty$ satisfy \Cref{lacunary assumption}, $L$ be an infinite subset of $\N$, and $(X_n)_{n=1}^\infty$ be a sequence of closed subspaces of $\ell_\infty$. When we write
\[\mathcal{Z}\big((m_j,n_j)_{j\in L},(X_n)_{n=1}^\infty\big),\]
we will mean $\mathcal{Z}\big((m'_j,n'_j)_{j=1}^\infty,(X_n)_{n=1}^\infty\big)$, where $(m_j',n_j')_{j=1}^\infty$ is relabeling of $(m_j,n_j)_{j\in L}$ given by the unique order preserving bijection $\N\mapsto L$.
\end{notation}

The following statement is a combination of \cite[Corollary 5.12 and Proposition 5.14]{zisimopoulou:2014}. The ``in particular'' part of \cref{RIS characterize horizontal compactness weight-missmatch} follows from \cite[Proposition 2.5]{argyros:haydon:2011}.

\begin{proposition}
\label{RIS characterize horizontal compactness}
Let $(m_j,n_j)_{j=1}^\infty$ satisfy \Cref{lacunary assumption}, $L$ be an infinite subset of $\N$, $(X_n)_{n=1}^\infty$ be a sequence of closed subspaces of $\ell_\infty$, and let $\mathcal{Z} = \mathcal{Z}\big((m_j,n_j)_{j\in L}, (X_n)_{n=1}^\infty\big)$. Then, there exists a class $\mathcal{C}_\mathrm{RIS}$ of bounded horizontally block sequences satisfying the following.
\begin{enumerate}[label=(\roman*),leftmargin=19pt]

    \item For every bounded linear operator $T:\mathcal{Z}\to Y$, where $Y$ is a Banach space, $T$ is horizontally compact if and only if for every $(y_k)_{k=1}^\infty$ in $\mathcal{C}_\mathrm{RIS}$, $\lim_k\|Ty_k\| = 0$.

    \item\label{RIS characterize horizontal compactness weight-missmatch} For every $(y_k)_{k=1}^\infty$ in $\mathcal{C}_\mathrm{RIS}$ there exists $C>0$ such that the map $e_k\mapsto y_k$ extends to a bounded linear operator of norm at most $C$ from the mixed-Tsirelson space $T[(\mathscr{A}_{3n_j},m^{-1}_j)_{j\in L}]$ to $\mathcal{Z}$.
    
    In particular, for every $j\in\N\setminus L$ and positive integers $k_1<\cdots<k_{n_{j}}$,
    \[\big\|\sum_{i=1}^{n_{j}}y_{k_i}\big\| \leq C\frac{n_j}{m_j^2}.\]

\end{enumerate}
\end{proposition}

The following is proved with essentially the same argument as in \cite[Lemma 10.3]{argyros:haydon:2011}.

\begin{proposition}
\label{main incomparability result}
Let $(m_j,n_j)_{j=1}^\infty$ satisfy \Cref{lacunary assumption} and let $L_1$, $L_2$ be subsets of $\N$ such that $L_1\setminus L_2$ and $L_2\setminus L_1$ are infinite. Let $(X^{(1)}_n)_{n=1}^\infty$, $(X^{(2)}_n)_{n=1}^\infty$ be sequences of closed subspaces of $\ell_\infty$. For $i=1,2$, let $\mathcal{Z}^{(i)} = \mathcal{Z}\big((m_j,n_j)_{j\in L_i},(X_n^{(i)})_{n=1}^\infty\big)$ and denote its associated sequence of projections $(P_n^{(i)})_{n=1}^\infty$. For a bounded linear operator $T:\mathcal{Z}^{(1)}\to\mathcal{Z}^{(2)}$, the following hold.
\begin{enumerate}[label=(\roman*),leftmargin=19pt]
    
    \item\label{main incomparability result i} If, for every $n\in\N$, $T P_{\{n\}}^{(1)}$ is compact then $\lim_n\|T-P_n^{(2)}T\| = 0$.

    \item\label{main incomparability result ii} If, for every $n\in\N$, $ P_{\{n\}}^{(2)}T$ is compact then $\lim_n\|T-TP_n^{(1)}\| = 0$, i.e., $T$ is horizontally compact.
    
\end{enumerate}

\end{proposition}

\begin{proof}
We will only prove the first assertion; the second one is similar and somewhat simpler. Denote $\mathcal{C}^{(1)}_\mathrm{RIS}$ the class of bounded horizontally block sequences in $\mathcal{Z}^{(1)}$ given by \Cref{RIS characterize horizontal compactness}.
\begin{claim}
For every $(y_k)_{k=1}^\infty$ in $\mathcal{C}^{(1)}_\mathrm{RIS}$, $\lim_n\sup_{k,m\geq n}\|P^{(2)}_{(m,\infty)}Ty_k\| = 0$.
\end{claim}
We will first use the claim to deduce $\lim_n\|P^{(2)}_{(n,\infty)}T\| = 0$, as desired. Assume this is false. Because, for each $n\in\N$, $TP_{\{n\}}^{(1)}$ is compact it easily follows that, for $n\in\N$, $\lim_k\|P_{(k,\infty)}^{(2)}TP_n^{(1)}\| = 0$. Therefore, for some $\e>0$ and for all $n\in\N$, $\limsup_k\|P^{(2)}_{(k,\infty)}TP^{(1)}_{(n,\infty)}\| >\e$. It is thus easy  find a bounded horizontally block sequence $(z_n)_{n=1}^\infty$ and a strictly increasing sequence of positive integers $(k_n)_{n=1}^\infty$ such that, for all $n\in\N$, $\|P^{(2)}_{(k_n,\infty)}Tz_n\| > \e/2$. For every $n\in\N$, fix $f_n$ in the unit ball of $\mathcal{Z}^{(2)}$ such that $|f_n(P^{(2)}_{(k_n,\infty)}Tz_n)| > \e/2$ and define the bounded linear operator $S:\mathcal{Z}^{(1)}\to c_0$ given by $Sx = (f_n(P^{(2)}_{(k_n,\infty)}Tx))_{n=1}^\infty$. This is well defined because $P_{(k_n,\infty)}T$, $n\in\N$, converges to zero in the strong operator topology. The sequence $(z_n)_{n=1}^\infty$ witnesses the non-horizontal compactness of $S$, therefore, there exists $(y_m)_{m=1}^\infty$ in $\mathcal{C}^{(1)}_\mathrm{RIS}$ such that $\limsup_m\|Sy_m\| >0$. Because the Schauder decomposition of $\mathcal{Z}^{(1)}$ is shrinking, $(y_m)_{m=1}^\infty$ is weakly null. Therefore, we deduce $\lim_n\sup_{\ell,m\geq n}|f_\ell(P_{(k_\ell,\infty)}^{(2)}Ty_m)| > 0$, and in particular,
\[\lim_n\sup_{k,m\geq n}\|P_{(k,\infty)}^{(2)}Ty_m  \| > 0,\]
contradicting the claim.

We proceed to prove the claim now. Fix a sequence $(y_k)_{k=1}^\infty$ in $\mathcal{C}^{(1)}_\mathrm{RIS}$ and let $C$ be the constant given by \Cref{RIS characterize horizontal compactness} \ref{RIS characterize horizontal compactness weight-missmatch}. Assume, towards contradiction, that, for some $\e>0$, $\lim_n\sup_{m,k\geq n}\|P^{(2)}_{(m,\infty)}y_k\| > \e$. We will prove that, for arbitrary $j\in L_2\setminus L_1$, that has even position in the increasing order of $L_2$, $\|T\| \geq \e m_j/(4C)$, which, by \Cref{lacunary assumption} would be absurd. We denote $\Gamma$ the set associated with the construction of $\mathcal{Z}^{(2)}$. After passing to a subsequence of $(y_k)_{k=1}^\infty$ and relabeling, there are $\eta_k\in\Gamma$, $j\leq p_k<q_k<p_{k+1}-1\in\N$, $\lambda_k\in G_{q_k}$, $k\in\N$, such that
\begin{equation*}
\mathfrak{Re}\Big(\lambda_ke_{\eta_k}^*\big(P^{(2)}_{(p_k,q_k]}Ty_k\big)\Big) >\e \text{ and } \big\|P_{(q_k,\infty)}Ty_k\big\| < \frac{\e}{6n^2_{j}}.
\end{equation*}
Because $(y_k)_{k=1}^\infty$ is weakly null, we may further assume that, for every $k\in\N$,
\[\max_{1\leq i<k}\Big|e_{\eta_i}^*\big(P^{(2)}_{(q_i,p_i]}Ty_k\big)\Big| < \frac{\e m_{j}}{6n^2_{j}}\text{ and }\max_{\xi\in\Gamma_{q_{k-1}+1}}|d_\xi^*(Ty_k)| < \frac{\e}{6n^2_{j}}.\]

Because $j$ has even position in $L_2$, by \Cref{coordinate builder} there exists $\gamma\in\Gamma$ with evaluation analysis
\[e_\gamma^* = \frac{1}{m_j}\sum_{k=1}^{n_j}\lambda_ke_{\eta_k}^*\circ P_{(p_k,q_k]} + \sum_{k=1}^{n_j}d_{\xi_k}^*,\]
for some $\xi_k\in\Delta_{q_k+1}^1$, $1\leq k\leq n_j$. We use this $\gamma$ to bound from below the norm of $T$.
\begin{equation}\label{main incomparability result eq1}
\begin{split}
    \|T\|C\frac{n_{j}}{m^2_{j}} &\geq \|T\|\Big\|\sum_{k=1}^{n_{j}}y_k\Big\|  \geq \Big|e_\gamma^*\Big(\sum_{k=1}^{n_{j}}Ty_k\Big)\Big|\\
&\geq \underbrace{\Big|\sum_{k=1}^{n_{j}}e_\gamma^*\Big(P^{(2)}_{(p_k,q_k]}Ty_k\Big)\Big|}_{=:\alpha} - \sum_{k=1}^{n_{j}}\underbrace{\Big|e_\gamma^*\Big(P^{(2)}_{[1,q_{k-1}+1]}Ty_k\Big)\Big|}_{=:\beta_k}\\
&\phantom{geq}- \sum_{k=1}^{n_{j}}\underbrace{\Big|e_\gamma^*\Big(P^{(2)}_{(q_{k-1}+1,p_k]}Ty_k\Big)\Big|}_{=0} - \sum_{k=1}^{n_{j}}\underbrace{\Big|e_\gamma^*\Big(P^{(2)}_{(q_k,\infty)}Ty_k\Big)\Big|}_{\leq \e/(6n_{2j})}.
\end{split}
\end{equation}
We estimate $\alpha$ and the $\beta_k$ as follows:
\begin{equation}
\label{main incomparability result eq2}
\alpha \geq \mathfrak{Re}\Big(\frac{1}{m_{j}}\sum_{k=1}^{n_{j}}\lambda_ke_{\eta_k}^*\big(P^{(2)}_{(p_k,q_k]}Ty_k)\Big) > \e\frac{n_{j}}{m_{j}},
\end{equation}
and, for $1\leq k\leq n_{j}$,
\begin{equation}
\label{main incomparability result eq3}
    \beta_k =\Big|\frac{1}{m_{j}}\sum_{i=1}^{k-1}\lambda_ie_{\eta_i}^*\Big(P_{(q_i,p_i]}Ty_k\Big) + \sum_{i=1}^{k-1}d_{\xi_i}^*(Ty_k)\Big| < \frac{\e}{6n_{j}} + \frac{\e}{6n_{j}}.
\end{equation}
Combining \eqref{main incomparability result eq1}, \eqref{main incomparability result eq2}, and \eqref{main incomparability result eq3}, we deduce $\|T\| \geq \e m_j/(4C)$.
\end{proof}

\section{The space $\mathcal{Z}_{\mathcal{K}(c_0)}$}
\label{ZKC0 section}
In this brief penultimate section we define the space $\mathcal{Z}_{\mathcal{K}(c_0)}$ and prove \Cref{AH sum of AH}. Fix a sequence of pairs $(m_j,n_j)_{j=1}^\infty$ satisfying \Cref{lacunary assumption} and let $L_1$, $L_2$ be disjoint infinite subsets of $\N$. Define
\[\mathfrak{X}_\mathrm{AH} = \mathcal{Z}\big((m_j,n_j)_{j\in L_1},(\{0\})_{n=1}^\infty\big),\]
i.e., the Argyros-Haydon sum of infinitely many copies of the zero subspace with parameters $(m_j,n_j)_{j\in L_1}$. This Banach space is separable, because its standard Schauder decomposition $(\mathcal{Z}_n)_{n=1}^\infty$ consists of finite dimensional spaces. This space is practically the same as the original Argyros-Haydon space from \cite{argyros:haydon:2011}, but for completeness, we justify its main property.

\begin{proposition}
\label{XAH has scalar-plus-compact}
The space $\mathfrak{X}_\mathrm{AH}$ has the scalar-plus-compact property.    
\end{proposition}

\begin{proof}
For each $n\in\N$, every $T:\mathcal{Z}\to X_n = \{0\}$ is, trivially, horizontally compact. Therefore, by \Cref{scalar-plus-horizontally compact zisimopoulou}, every bounded linear operator is a scalar multiple of the identity plus a horizontally compact operator. But, because the standard Schauder decomposition is finite dimensional, horizontally compact operators are in fact compact.
\end{proof}

We fix, for each $n\in\N$, an isometric copy $X_n$ of $\mathfrak{X}_{AH}$ in $\ell_\infty$ and define
\[\mathcal{Z}_{\mathcal{K}(c_0)} = \mathcal{Z}\big((m_j,n_j)_{j\in L_2},(X_n)_{n=1}^\infty\big).\]

\begin{proposition}
\label{ZKC0 has scalar-plus-horizontally approximable}
The space $\mathcal{Z}_{\mathcal{K}(c_0)}$ has the scalar-plus-horizontally approximable property.    
\end{proposition}

\begin{proof}
We first argue that $\mathcal{Z}_{\mathcal{K}(c_0)}$ has the scalar-plus-horizontally compact property. By the finite-dimensionality of the standard Schauder decomposition of $\mathfrak{X}_\mathrm{AH}$, for every bounded linear operator $T:\mathcal{Z}_{\mathcal{K}(c_0)}\to \mathfrak{X}_\mathrm{AH}$ the assumption of \Cref{main incomparability result} \ref{main incomparability result ii} is satisfied, and therefore, $T$ is horizontally compact. By \Cref{scalar-plus-horizontally compact zisimopoulou}, every bounded linear operator $T:\mathcal{Z}_{\mathcal{K}(c_0)}\to \mathcal{Z}_{\mathcal{K}(c_0)}$ is a scalar multiple of the identity plus a horizontally compact operator.

We will next show that horizontally compact operators on $\mathcal{Z}_{\mathcal{K}(c_0)}$ are horizontally approximable. Let $T:\mathcal{Z}_{\mathcal{K}(c_0)}\to\mathcal{Z}_{\mathcal{K}(c_0)}$ be horizontally compact; we will show $\lim_n\|P_{(m,\infty)}T\| = 0$. If this is false, because $T = \lim_nTP_n$, there is $n_0\in\N$ such that $\limsup_m\|P_{(m,\infty)}TP_{n_0}\| > 0$. Then, for some $1\leq n\leq n_0$, $\limsup_m\|P_{(m,\infty)}TP_{\{n\}}\|>0$. Because $\mathcal{Z}_n$ is isomorphich to $\mathfrak{X}_\mathrm{AH}\oplus\ell_\infty(\Delta_n)$, and $\Delta_n$ is finite, there exists a bounded linear operator $S:\mathfrak{X}_\mathrm{AH}\to\mathcal{Z}_{\mathcal{K}(c_0)}$ such that $\liminf_m\|P_{(m,\infty)}S\| > 0$. But, by the finite-dimensionality of the Schauder decomposition of $\mathfrak{X}_\mathrm{AH}$, $S$ satisfies the assumption of \Cref{main incomparability result} \ref{main incomparability result ii} and, therefore, $\liminf_m\|P_{(m,\infty)}S\| = 0$.
\end{proof}

\begin{remark}
By \Cref{shrinking and duality}, the dual of $\mathcal{Z}_{\mathcal{K}(c_0)}$ is 4-isomorphic to $\ell_1$. In particular, $\mathcal{Z}_{\mathcal{K}(c_0)}$ has the approximation property, and thus, $\mathcal{K}(\mathcal{Z}_{\mathcal{K}(c_0)})$ is the minimum non-trivial closed ideal of $\mathcal{L}(\mathcal{Z}_{\mathcal{K}(c_0)})$. Because $\mathcal{K}(c_0)$ is simple, it follows that the only non-trivial closed ideals of $\mathcal{L}(\mathcal{Z}_{\mathcal{K}(c_0)})$ are $\mathcal{K}(\mathcal{Z}_{\mathcal{K}(c_0)})$ and $\mathcal{HA}(\mathcal{Z}_{\mathcal{K}(c_0)})$.
\end{remark}

\section{Open problems and directions}
We share some open problems and directions that arise naturally from this work. For additional questions of a similar nature we refer the reader to \cite[Section 8]{motakis:puglisi:tolias:2020}, \cite[Section 9]{motakis:2024}, and \cite[Section 3]{motakis:pelczar:2024}.

A combination and refinement of techniques from \cite{motakis:puglisi:zisimopoulou:2016} and this paper has potential to answer the following.
\begin{problem}
Let $K$ be a countable compact metric space. Is the unitization of $\mathcal{K}(C(K))\equiv C(K;\ell_1(K))$ isomorphic as a Banach algebra to the Calkin algebra of some Banach space?
\end{problem}

 Let $1<p<\infty$ and let $q$ denote its conjugate exponent. In the definition of a Bourgain-Delbaen-$\mathscr{L}_\infty$-sum, one can replace the space $\mathcal{W}$ with the space $\big(\big(\oplus_{k=1}^\infty X_n\big)_p\oplus \big(\oplus_{n=1}^\infty\ell_\infty(\Delta_n)\big)_\infty\big)_\infty$ to achieve a mixed Bourgain-Delbaen-$(p,\infty)$-sum of a sequence of Banach spaces $(X_n)_{n=1}^\infty$ (see \cite{motakis:puglisi:tolias:2020} for a similar mixed sum). The endeavour of studying the properties of such a sum may be useful in studying the following. Let $X$ be a infinite dimensional Banach space with trivial type (such as $c_0$ or $\ell_1$), and let $Y = \ell_p(X)$, considered with its standard Schauder decomposition. Consider the subalgebra of $\mathcal{HA}(Y)$ of all horizontally approximable operators with matrix representation $A = (a_{m,n}I_{X})_{m,n=1}^\infty$. Then,
\[
\begin{split}
\|A\|_{\mathcal{L}(Y)} &= \sup\Big\{\sum_{m=1}^\infty\sum_{n=1}^\infty |x_ny_ma_{m,n}|: \|(x_n)_{n=1}^\infty\|_p\|(y_m)_{m=1}^\infty\|_q\leq 1\Big\}\\
& = \|(|a_{m,n}|)_{m,n=1}^\infty\|_{\mathcal{L}(\ell_p)}.
\end{split}
\]
Put differently, this space is the completion of the space of finitely supported scalar matrices in $\mathcal{K}(\ell_p)$ with the smallest norm dominating $\|\cdot\|_{\mathcal{L}(\ell_p)}$ making the collection $E_{i,j} = (\delta_i(m)\delta_j(n))_{m,n=1}^\infty$, $i,j\in\N$, a 1-unconditional basis. Let us denote this Banach algebra $\mathcal{K}(\ell_p)_\mathrm{u}$.

\begin{problem}
\label{unconditionalization of compacts}
For $1<p<\infty$, is $\mathcal{K}(\ell_p)_\mathrm{u}$ isomorphic as a Banach algebra to the Calkin algebra of some Banach space?
\end{problem}

This can also be asked for the analogously defined space $\mathcal{K}(V)_\mathrm{u}$, for any Banach space $V$ with an unconditional basis. A mixed Bourgain-Delbaen-$(V,\infty)$ is the natural candidate for solving this if $V$ does not contain $\ell_1$. Other types of mixed Bourgain-Delbaen sums based on those in \cite{motakis:puglisi:tolias:2020} may be able to address the case when $\ell_1$ is a subspace of $V$.

\Cref{unconditionalization of compacts} can be seen as a variation of the more challenging goal of providing a representation of an explicit infinite dimensional non-commutative $C^*$-algebra, such as $\mathcal{K}(\ell_2)$, as a  Calkin algebra (see \cite[Section 9]{motakis:2024} for a detailed discussion on related problems posed by N. C. Phillips). Some far simpler known examples are those in \cite[Corollary 8.8]{motakis:2024}, and it was shown in \cite{kania:laustsen:2017} that every finite-dimensional semi-simple complex algebra is a Calkin algebra. 

\begin{problem}
Is the unitization of $\mathcal{K}(\ell_2)$ isomorphic as a Banach algebra to the Calkin algebra of some Banach space?
\end{problem}

More generally, if $X$ is a Banach space with a shrinking Schauder basis, is $\mathcal{K}(X)$ isomorphic to the Calkin algebra of some Banach space?

\bibliographystyle{plain}
\bibliography{bibliography}

\end{document}